\documentclass[preprint,12pt]{elsarticle}

\usepackage{amssymb}
\usepackage{amsmath}
\usepackage{amsthm}

\theoremstyle{plain} 
\newtheorem*{theorem*}{Theorem}

\journal{Applied Mathematics Letters}

\begin{document}

\begin{frontmatter}



\title{Generalizing the Eigenvalue Interlacing Theorem to Pseudo-Similarity Transformations}


\author[inst2]{Julio Guillen-Garcia\corref{cor1}}
\ead{julio.guillen@urjc.es}

\author[inst1]{Manuel F. Fernández}

\author[inst2]{Roberto Gallardo-Cava}

\cortext[cor1]{Corresponding author.}

\affiliation[inst2]{%
  organization={Rey Juan Carlos University},
  city={Móstoles},
  state={Madrid},
  country={Spain}
}

\affiliation[inst1]{%
  organization={Independent Researcher},
  city={Syracuse},
  state={New York},
  country={USA}
}

\begin{abstract}
The current general form of the well-known Eigenvalue Interlacing Theorem states that, given an $N \times N$ Hermitian matrix $P$, the eigenvalues of the matrix product $Q^HPQ$ will “interlace” those of $P$ if the columns of $N \times L$ matrix $Q$, $L \leq N$, are unitary. This note further generalizes such theorem to include pseudo-similarity transformations; namely, products of the form $H^{\dagger}PH$, where $H$ is a general $N \times K$ matrix and ``$\dagger$'' represents the pseudo-inverse operation. This implies that, while the product $Q^HPQ$ is also Hermitian and is generally a deflated version of $P$, both in dimensionality and in the number of non-zero eigenvalues, such will not the case with $H^{\dagger}PH$, as this product, although generally a deflated version of $P$ in the number of non-zero eigenvalues, will not necessarily be so in dimensionality, nor will it generally be Hermitian. Thus, this note not only further generalizes the Eigenvalue Interlacing Theorem but also shows that eigenvalue interlacing may occur among Hermitian and non-Hermitian matrices and even with dimensionally-inflated matrices. 
\end{abstract}


\begin{highlights}
\item Introduces and formally defines the concept of pseudo-similarity transformation of a matrix.
\item Extends the Eigenvalue Interlacing Theorem to Hermitian matrices under pseudo-similarity transformations.
\item Shows that pseudo-similarity transformations yield eigenvalue interleaving between Hermitian and non-Hermitian matrices.
\item Demonstrates that pseudo-similarity transformations also produce eigenvalue interleaving between Hermitian and dimensionally-inflated matrices.
\end{highlights}


\end{frontmatter}



In its most basic (Cauchy's) form, the Eigenvalue Interlacing Theorem states that, given an $N \times N$ Hermitian matrix $P$ in $\mathbb{C}$ (to be assumed here, without lack of generality, of full rank $N$), with eigenvalues\footnote{The eigenvalues of a Hermitian matrix are real \cite{golubvanloan1996}.} arranged in non-decreasing order $\lambda_1 \leq \lambda_2 \leq \dots \space \lambda_{N-1} \leq \lambda_N$, and given an $N \times L$, $L \leq N$, matrix $\hat{I}$ composed of $L$ independent unit vectors (vectors of zeros except for a $1$ in the nth entry, with no two vectors having $1$s in the same location; i.e., $\hat{I}^T\hat{I} = I_{L \times L}$), we will have \cite{hwang2004}

\begin{equation}
\lambda_l \leq \mu_l \leq \lambda_{N-L+l}, \quad l = 1, 2, \space \dots, L  
\end{equation}
where the $\mu_1 \leq \mu_2 \leq \dots \space \mu_{L-1} \leq \mu_L$ , are the eigenvalues, in non-decreasing order, of compressed or “deflated” $L \times L$ matrix

\begin{equation} \label{eq:2}
T = \hat{I}^TP\hat{I} \space .
\end{equation}

Extending this property, from $\hat{I}$ to general unitary matrices, is simple enough:  Since the eigenvalues of a matrix are unaffected by similarity transformations, given an $N \times N$ unitary matrix $\hat{Q}$, the eigenvalues of $P$ will equal those of $\hat{Q}^HP\hat{Q}$.  Hence, the eigenvalues of

\begin{equation}
\hat{T} = \hat{I}^T \hat{Q}^HP\hat{Q} \hat{I} \space  ,   
\end{equation}
although not necessarily equal to those of $T$ in (\ref{eq:2}), will also interlace those of $P$. In other words, the eigenvalues of the compressed matrix

\begin{subequations}
\begin{equation}\label{eq:4a}
\hat{T} = Q^HPQ \space ,
\end{equation}
where
\begin{equation}\label{eq:4b}
Q = \hat{Q}\hat{I} \space ,
\end{equation}
\end{subequations}
will interlace those of $P$ \cite{bellman1970}.

A question would be whether this result could be generalized to the case where, instead of $\hat{Q}$ and $\hat{Q}^H$, one may use any full-rank, non-unitary $N \times N$ matrix $X$ and its inverse, as these may be used to similarity-transform $P$. That is, since $P$ and $X^{-1}PX$ have the same eigenvalues \cite{golubvanloan1996}, would the eigenvalues of  $\hat{I}^TX^{-1}PX\hat{I}$  interlace those of $P$?  The general answer to this question is “NO”, because though $P$ is Hermitian, $X^{-1}PX$ will not be so; hence, the interlacing theorem will not generally apply. Therefore, the Eigenvalue Interlacing Theorem will not generally hold under oblique projections (i.e., projections of the general form $X\hat{I}\hat{I}^T X^{-1}$).

However, there’s another way of further generalizing the interlacing property; namely, given an $N \times N$ Hermitian matrix $P$, with its eigenvalues again arranged in non-decreasing order $\lambda_1 \leq \lambda_2 \leq \dots \space \lambda_{N-1} \leq \lambda_N$, and given a general $N \times K$ matrix $H$ of rank $L$, we will have

\begin{equation} \label{eq:5}
\lambda_l \leq \eta_l \leq \lambda_{N-L+l}, \quad l = 1, 2, \space \dots, L  	
\end{equation}
where the $\eta_1 \leq \eta_2 \leq \dots \space \eta_{L-1} \leq \eta_L$ , are the eigenvalues, in non-decreasing order, of “pseudo-similarly transformed” $K \times K$ matrix

\begin{equation} \label{eq:6}
\mathcal{T} = H^{\dagger} P H \space ,  	
\end{equation}
where the superscripted “$\dagger$” is used to denote the pseudo-inverse \cite{golubvanloan1996} of $H$. More formally:

\begin{theorem*}[``Eigenvalue interlacing under pseudo-similarity transformations'']\label{thm:main}
Given an $N \times N$ Hermitian matrix $P$ and a general $N \times K$ matrix $H$ of rank $L$, both in $\mathbb{C}$, the $L$ eigenvalues of the $K \times K$ matrix $T$ resulting from the pseudo-similarity transformation $T = H^{\dagger} P H$ will interlace those of $P$.
\end{theorem*}

\begin{proof} For simplicity of presentation we will begin by assuming an $N \times L$ matrix $H$, $L \leq N$, of full-rank $L$ (these assumptions will be removed later). The “economy” $QR$-decomposition \cite{golubvanloan1996} of such matrix is

\begin{equation} \label{eq:7}
H = QR \space ,  	
\end{equation}
columns and $L \times L$ matrix $R$ is upper-triangular; so, its pseudo-inverse will be \cite{golubvanloan1996}

\begin{equation} \label{eq:8}
H^{\dagger} = R^{-1}Q^H  \space ,
\end{equation}

Expressions (\ref{eq:7}) and (\ref{eq:8}) mean that (\ref{eq:6}) may be written as

\begin{equation} \label{eq:9}
\mathcal{T} = R^{-1}(Q^H PQ)R \space ,
\end{equation}
where the eigenvalues $\eta_1 \leq \eta_2 \leq \dots \space \eta_{L-1} \leq \eta_L$  of $L \times L$ deflated matrix $Q^HPQ$ would satisfy (\ref{eq:5}).  And since matrix $\mathcal{T}$ is a similarity-transformed version of $Q^HPQ$, it will have the same eigenvalues as the latter, meaning that the eigenvalues of (\ref{eq:9}) will also satisfy (\ref{eq:5}). 

Let’s now eliminate the full-rank requirement. Using the same matrices presented above, this can be done, without lack of generality, by constructing a rank-deficient $N \times K$ matrix $\tilde{H}$ as:

\begin{equation}
\tilde{H} = HV^H = QRV^H  \space ,    
\end{equation}
where $V$ is a $K \times L$, $K > L$, matrix of unitary columns, and $Q$ and $R$ are as defined above. Then, 

\begin{equation}
 \tilde{T} = \tilde{H}^{\dagger} P \tilde{H} = V (R^{-1}Q^H P Q R) V^H = V \mathcal{T} V^H \space .  
\end{equation}

Since $V$ is unitary, the $K$ eigenvalues of $\tilde{T}$ will consist of the $L$ eigenvalues of $\mathcal{T}$ and $K-L$ zeros\footnote{This can be shown by inserting matrix $\mathcal{T}$ as, let’s say, the top-left $L \times L$ portion of a $K \times K$ matrix $\hat{T}$ containing zeros everywhere else. Unitary matrix $V$ could also be conceptualized as composing the first $L$ columns of a $K \times K$ unitary matrix $\hat{V}$. Then, the $K$ eigenvalues of similarity transformation matrix $\hat{V} \hat{T} \hat{V}^H$ would clearly equal those of $\hat{T}$, which would in turn equal the $L$ nonzero eigenvalues of $\mathcal{T}$ with $K-L$ additional zeros.}; hence, the $L$ nonzero eigenvalues of $\tilde{T}$ will interlace the $N$ eigenvalues of $P$.

Observe that this result also extends to the case where $K \geq N$, meaning that the resulting $\tilde{T}$ may be of dimension larger than that of $P$. \end{proof}

\section*{Conclusions}
This short correspondence first introduced and defined the concept of the pseudo-similarity transformation of a matrix. Namely, given an $N \times N$ matrix $P$ and a general $N \times K$ matrix $H$, both in $\mathbb{C}$, the product $H^{\dagger}PH$ was defined as a pseudo-similarity transformation of $P$.  

It was then shown that, when matrix $P$ is Hermitian, the nonzero eigenvalues of the $K \times K$ pseudo-similarly transformed matrix $H^{\dagger}PH$ will interlace those of $P$.  This result constitutes a generalization than encompasses the current version of the Eigenvalue Interlacing Theorem, as the latter involves the case where $P$ is transformed by a column-unitary matrix $Q$ as in expression (\ref{eq:4a}), and since the pseudo-inverse of such $Q$ is $Q^H$, expression (\ref{eq:4a}) is subsumed in (\ref{eq:6}).

Interesting aspects of this result are: (a) the resulting matrix $H^\dagger PH$ will not generally be Hermitian, and (b) since the number $K$ of columns of $H$ is independent of any consideration (i.e., it does not depend on the dimensionality of $N$ nor on the rank of $H$), matrix $H^\dagger PH$, although perhaps deflated in rank, may be of larger dimensionality than that of $P$.  These results thus imply that, under certain conditions, the eigenvalues of non-Hermitian and of dimensionally-inflated matrices may interlace those of Hermitian ones. We believe these results will be useful not only to researchers in the various areas of applied mathematics (matrix analysis, graph or spectral theory, etc.) but also of practical interest to scientists and engineers in fields such as fluid- and thermo-dynamics, signal processing, systems and control, optimization theory, AI, etc.

\section*{Declaration of Competing Interest}
The authors declare that they have no known competing financial interests or personal relationships that could have appeared to influence the work reported in this paper.

\section*{Funding}
This research did not receive any specific grant from funding agencies in the public, commercial, or not-for-profit sectors.

\end{document}